\documentclass[12pt,reqno]{amsart}

\usepackage{amsaddr}
\usepackage{amssymb,  mathrsfs}
\usepackage{tikz-cd}

\usepackage[utf8]{inputenc}

\usepackage{a4wide}
\usepackage{enumerate}
\usepackage{tikz-cd}   
\usepackage{graphicx}
\usepackage{tikz}
\usepackage{float}
\usepackage{wrapfig}

\newcommand{\N}{\mathbb{N}}
\newcommand{\Z}{\mathbb{Z}}

\newcommand{\R}{\mathbb{R}}

\DeclareMathOperator{\End}{End}

\DeclareMathOperator{\im}{im}
\DeclareMathOperator{\id}{id}
\DeclareMathOperator{\Supp}{Supp}
\DeclareMathOperator{\Span}{Span}
\DeclareMathOperator{\degree}{deg}
\DeclareMathOperator{\Char}{char}
\DeclareMathOperator{\Rank}{rank}

\title{Very good gradings on matrix rings are epsilon-strong}
\author{Patrik Lundstr\"om}
\address{Department of Engineering Science,
University West, SE-46186 Trollhättan, Sweden}
\email{patrik.lundstrom@hv.se}

\author{Johan \"Oinert }
\address{Department of Mathematics and Natural Sciences,
Blekinge Institute of Technology,
SE-37179 Karlskrona, Sweden}
\email{johan.oinert@bth.se}

\author{Laura  Orozco}
\address{Escuela de Matematicas, Universidad Industrial de Santander, Cra. 27 Calle 9  UIS
	Edificio 45, Bucaramanga, Colombia} \email{lnorogar@correo.uis.edu.co}
	
\author{H\'ector  Pinedo}
\address{Escuela de Matematicas, Universidad Industrial de Santander, Cra. 27 Calle 9  UIS
	Edificio 45, Bucaramanga, Colombia} 
	\email{hpinedot@uis.edu.co}


	
	\date{\today}

\usepackage{amsmath,amssymb,amsthm}

\theoremstyle{plain}
\newtheorem{theorem}{Theorem}[section]
\newtheorem{lemma}[theorem]{Lemma}
\newtheorem{prop}[theorem]{Proposition}
\newtheorem{cor}[theorem]{Corollary}
\theoremstyle{definition}
\newtheorem{defn}[theorem]{Definition}

\newtheorem{exmp}[theorem]{Example}

\theoremstyle{remark}
\newtheorem{remark}[theorem]{Remark}

\begin{document}

\maketitle

\begin{abstract}
We investigate properties of group gradings on matrix rings $M_n(R)$, where $R$ is an associative unital ring and $n$ is a positive integer. More precisely, we introduce very good gradings and show that any very good grading on $M_n(R)$ is necessarily epsilon-strong.
We also identify 
a condition that is sufficient to guarantee that $M_n(R)$ is an epsilon-crossed product, i.e. isomorphic to a crossed product associated with a unital twisted partial action.
In the case where $R$ has IBN, we are able to provide a characterization of when $M_n(R)$ is an epsilon-crossed product.
Our results are illustrated by several examples.
\end{abstract}

\section{Introduction}

Suppose that $G$ is a group, $n$ is a positive
integer, $K$ is a field, and let
$M_n(K)$ denote the ring of all $n \times n$ 
matrices with entries in $K$.
A classical problem in algebra, supposedly 
formulated by Zelmanov \cite[Problem 1]{kelarev1995},
asks for a description of all $G$-gradings on 
$M_n(K)$, that is, to describe all
decompositions into $K$-vector subspaces 
$M_n(K) = \oplus_{g \in G} M_n(K)_g$ 
such that for all $g,h \in G$ the inclusion 
$M_n(K)_g M_n(K)_h \subseteq M_n(K)_{gh}$ 
holds.
This problem has attracted a lot of attention
(see e.g.
\cite{BSZ}--\cite{lundstrom2014} and 
\cite{val02}--\cite{zai01}) and it has been
solved for particular 
classes of fields $K$ and groups $G$, but,
to the best knowledge of the authors
of the present article, the general case
is still unsettled.
However, if we restrict ourselves to 
finding all 
\emph{good} $G$-gradings on $M_n(K)$, 
then a complete answer is given 
in \cite{das99}. 
Namely, following the
notation introduced in loc. cit.,
a $G$-grading on $M_n(K)$ is called good 
if each $M_n(K)_g$, for $g \in G$,
considered as a left vector space over $K$, 
has a basis consisting of matrices of the
type $e_{i,j}$ with 1 in position $(i,j)$ and zeros in every other position. 
In \cite[Prop.~2.1]{das99}, it was shown 
that there is a bijective correspondence
between the set of all maps 
$f : \{ 1,\ldots,n-1 \} \to G$ and the set 
of all good $G$-gradings on $M_n(K)$,
where each such $f$ is associated with the 
good grading on $M_n(K)$ induced by the
requirement that $e_{i,i+1} \in M_n(K)_{f(i)}$, for 
$i \in \{1,\ldots,n-1\}$.
This can then be used to obtain 
criteria for when a given 
$G$-grading on $M_n(K)$ is strong or a
$G$-crossed product (see
\cite[Prop.~2.3 and Prop.~2.6]{das99}). 
In \cite{das2001}, a bijection, similar to
the one described above was 
obtained in the more general case where 
the grading is defined by a semigroup, 
the bijection in that
case depending on the structure of all simple and zero simple principal factors in
the semigroup.
Analogous problems for magma gradings (and
filtrations) on rings have also 
been considered in \cite{lundstrom2014}. 
Good  gradings have also been defined for other 
types of algebras such as subalgebras of matrix algebras over more general algebraic structures, incidence algebras \cite{jones, price,priceCorr}, and algebras with multiplicative bases \cite{BZS}.

In this article, we wish to generalize
some of the results mentioned above to
the case of matrix rings over 
arbitrary associative unital rings. 
To that end, we choose to distinguish between \emph{good} and \emph{very good} gradings (see Definition~\ref{def:GoodVeryGood}).
Noteworthily, 
we obtain our results 
using novel methods, namely by
utilizing results from the theory of
so-called \emph{epsilon-strongly graded rings} (see e.g. \cite{NYO}).
Our first main result (see Theorem~\ref{thm:main}) asserts that any very good grading on a matrix ring is necessarily \emph{epsilon-strong}.
Moreover, for such gradings our second main result provides sufficient conditions ensuring that the matrix ring is an \emph{epsilon-crossed product} (see Theorem~\ref{thm:pre}).
Our third main result offers a complete characterization of when a matrix ring over an IBN-ring is an epsilon-crossed product (see Theorem~\ref{thm:epsiloncrossed}).
Recall that, by \cite[Thm. 35]{NYO}, any 
epsilon-crossed product is isomorphic to a crossed product by a unital twisted partial action.
Algebras associated with partial actions, for instance, crossed products by unital twisted partial actions,
is a young but very active research area. For an account of numerous results in this direction,
we refer the reader to \cite{Dokuchaev2019,ExelBook} and the references therein.

\section{Good and very good group gradings}\label{goode}

Throughout this article, 
$R$ and $S$ 
are associative unital
rings,
$n$ 
is
a positive integer, 
$\overline{n}$ denotes the set 
$\{ 1,\ldots,n \}$, and $G$ 
is a group with identity 
element $e$.

\subsection{Gradings of bases}
In this article, unless otherwise stated,
$V$ denotes a free
left $R$-module
with a fixed basis 
$\overline{v} :=
\{ v_i \}_{i \in \overline{n}}$.
By a \emph{grading of $\overline{v}$} 
we mean a
map $\degree : \overline{v} \to G$.
Note that such a map induces a 
$G$-grading on $V$ in the sense
that $V = \oplus_{g \in G} V_g$
as left $R$-modules, where 
$V_g := \sum_{v \in \degree^{-1}(g)} 
Rv$, for $g \in G$.
On the other hand, suppose that
we are given a $G$-grading 
$V = \oplus_{g \in G} V_g$
on $V$. 
We say that this grading 
\emph{respects
$\overline{v}$}
if 
for each $i \in I$ there is $g_i \in G$
such that $v_i \in V_{g_i}$.
In that case, 
we can define a grading $\degree$
of $\overline{v}$ by putting
$\degree(v_i) := g_i$.
Clearly, the above constructions
yield a bijective correspondence 
between the set of gradings on 
$\overline{v}$ and the set of
gradings on $V$ respecting 
$\overline{v}$.

\subsection{Graded rings}
For the rest of this article,
suppose that $S$ is 
\emph{$G$-graded} as a ring. 
Recall that this means that 
there for each $g \in G$
is an additive subgroup $S_g$ of $S$ such
that $S = \oplus_{g \in G} S_g$, as
additive groups, and
for all $g,h \in G$ the inclusion
$S_g S_h \subseteq S_{gh}$ holds.
If $s \in S_g \setminus \{ 0 \}$,
then we write $\degree(s) := g$.
The $G$-grading on $S$ is said to be \emph{strong} if $S_g S_h = S_{gh}$
for all $g,h \in G$. Recall that $S$ is 
said to be a \emph{$G$-crossed product} if, for every
$g \in G$, $S_g$ contains an element which
is invertible in $S$. If $S$ is
a $G$-crossed product, then $S$ is 
strongly $G$-graded 
(see \cite[Rem. 1.1.2]{nas04}).
Recall from \cite{NYO}, that $S$ is said to be
\emph{epsilon-strongly $G$-graded} if, for 
every $g \in G$, there is some $\epsilon_g \in S_g S_{g^{-1}}$
such that 
$\epsilon_g s = s \epsilon_{g^{-1}} = s$ 
for every $s \in S_g$.
Note that, in that case, $\epsilon_g$, for $g\in G$, is uniquely defined by the grading.

If $S$ is epsilon-strongly $G$-graded,
then $S$ is said to be an 
\emph{epsilon-crossed product}
if, for every $g \in G$, there are $s_g \in S_g$ 
and
$t_{g^{-1}} \in S_{g^{-1}}$ such that
$s_g t_{g^{-1}} = \epsilon_g$ and 
$t_{g^{-1}} s_g = \epsilon_{g^{-1}}.$ 
Note that any 
epsilon-crossed product is isomorphic to a unital partial crossed product (see \cite[Thm. 35]{NYO}).

\subsection{Good and very good gradings on matrix rings}
For the rest of this article,
$M_n(R)$ denotes the ring of all 
$n \times n$ 
matrices with entries in $R$.

\begin{defn}\label{def:GoodVeryGood}
Suppose that $S:=M_n(R)$ is equipped with a $G$-grading.
\begin{itemize}
    \item[(i)] The grading is said to be \emph{good}, if 
$e_{i,j}$ is homogeneous for all $i,j \in \overline{n}$.

    \item[(ii)] The grading is said to be \emph{very good}, if for all $i,j \in \overline{n}$ there is $g \in G$ with
    $R e_{i,j} \subseteq S_g$.
\end{itemize}
\end{defn}

\begin{exmp}\label{ex:ABgrading}
Suppose that $S:=M_n(R)$ is equipped with a 
$G$-grading. If this grading is very good,
then it is, of course, good.
However, the reversed implication does not
always hold. Indeed,
let $A$ and $B$ be groups, and endow $R$ with any nontrivial $A$-grading. To every element $e_{i,j}$, we assign a $B$-degree  such that $\degree(e_{i,j}) \degree(e_{j,k}) = \degree(e_{i,k})$ for every $k\in \overline{n}$. Then the ring $M_n(R)$ is in a natural way graded by the product group $A \times B$. Moreover, $e_{i,j}$ will be homogeneous (w.r.t. the $A \times B$-grading) for all $i,j \in \overline{n}$. However, $r e_{i,i}$ will be non-homogeneous (w.r.t. the $A\times B$-grading) for any $r \in R$ that is non-homogeneous (w.r.t. the $A$-grading). A concrete example of this
would be to consider the field $R = \mathbb{C}$ of complex numbers equipped with its natural $\mathbb{Z}_2$-grading. Put $\degree(e_{1,1})=\degree(e_{2,2})=0$ and $\degree(e_{1,2})=\degree(e_{2,1})=1.$ Then $M_2(\mathbb{C})$ is $\mathbb Z_2 \times \mathbb Z_2$-graded. But for instance $(1+i)e_{1,1}$ is non-homogeneous w.r.t. the $\mathbb Z_2 \times \mathbb Z_2$-grading.
\end{exmp}

The exact relationship between good and very
good gradings is explained in the following:

\begin{prop}
Suppose that $S:=M_n(R)$ is equipped with 
a good $G$-grading. Then the following assertions
are equivalent:
\begin{itemize}
    \item[(i)] the $G$-grading on $S$
    is very good;
    \item[(ii)] for every $i \in 
    \overline{n}$ the inclusion
    $R e_{i,i} \subseteq S_e$ holds;
     \item[(iii)] $R 1_S \subseteq S_e$.
\end{itemize}
\end{prop}

\begin{proof}
(i)$\Rightarrow$(ii): 
Suppose that (i) holds. Take $i \in 
\overline{n}$. Since $e_{i,i}$ is a 
nonzero idempotent it follows that
$\deg(e_{i,i}) = e$. From (i) it follows 
that $R e_{i,i} \subseteq S_e$.

(ii)$\Rightarrow$(iii):
Suppose that (ii) holds. Then
$R 1_S = R \sum_{i=1}^n e_{i,i} 
\subseteq \sum_{i=1}^n R e_{i,i} 
\subseteq S_e$. 

(iii)$\Rightarrow$(i): 
Suppose that (iii) holds.
Take $i,j \in \overline{n}$ and
put $g := \deg( e_{i,j} )$. By (iii)
we get that
$R e_{i,j} = (R 1_S) e_{i,j} \subseteq 
S_e S_g \subseteq S_g$.
\end{proof}

It is not hard to see that a 
very good $G$-grading on $M_n(R)$ is 
determined by the $n$-tuple 
$\overline{g} = \{ g_1,\ldots,g_n \}$
where $g_i := \degree(e_{i,1})$,
for $i \in \overline{n}$.
Note that $g_1^2 = \degree( e_{1,1}^2 ) = \degree(e_{1,1}) = g_1$
so that necessarily $g_1 = e$.
This very good grading defines a 
grading on $\overline{v}$ if 
we put $\degree(v_i) := g_i$ for
$i \in  \overline{n}$.
This, in turn, defines a grading
on the free left $R$-module $V:=R^n$.
Note that $\degree(v_1) = g_1 = e$.
Thus, in particular,
$e \in \Supp(V)$, where 
$\Supp(V):=\{g\in G\mid V_g\neq 0\}$.
For all $i,j \in \overline{n}$ we consider $e_{i,j}$
as an element of $\End_R(V)$ in the following way.
Take $i,j,k \in \overline{n}$. 
Define $e_{i,j}(v_j) := v_i$, and $e_{i,j}(v_k) := 0$,
if $k \neq j$. Note that $e_{i,j}$ is then also 
a graded linear map of degree $\degree(e_{i,j}) = g_i g_j^{-1}$
since $\degree( e_{i,j}(v_j) ) = \degree(v_i) = g_i =
g_i g_j^{-1} g_j = \degree(e_{i,j} ) \degree(v_j)$
for all $i,j \in \overline{n}$.

\begin{theorem}\label{thm:main}
Every very good $G$-grading on $M_n(R)$ is an 
epsilon-strong $G$-grading.
\end{theorem}

\begin{proof}
Put $S:=M_n(R)$ and suppose that it is equipped with a very good $G$-grading. For each $g \in G$, we put 
$\mathcal{L}_g:= \{ i \in \overline{n} \mid 
\text{there is some } j \in \overline{n} 
\text{ such that } e_{i,j} \in S_g\}$ 
and
$\mathcal{R}_g := \{ j \in \overline{n} \mid 
\text{there is some } i \in \overline{n}
\text{ such that } e_{i,j} \in S_g\}.$
Set $\epsilon_g := \sum_{i \in \mathcal{L}_g} e_{i,i}$
and $\epsilon_g' := \sum_{j \in \mathcal{R}_g} e_{j,j}$.
It is not difficult  to see that $\epsilon_g \in S_g S_{g^{-1}}$, $\epsilon_g' \in S_{g^{-1}} S_g$
and that
$\epsilon_g s = s \epsilon_g' = s$ 
for every $s\in S_g$.
Since $\mathcal{R}_g =\mathcal{L}_{g^{-1}}$ it follows that
$\epsilon_g' = \epsilon_{g^{-1}}$ and 
thus $S$ is epsilon-strongly $G$-graded.
\end{proof}

\begin{remark} 
By considering $M_n(R)$ as a so-called
\emph{Leavitt path algebra} over $R$ 
defined by a line
graph with $n$ vertices (see 
\cite[Example~1.4(i)]{AAPI} for such 
a presentation when $R$ is a field)
and using \cite[Thm.~4.1]{NO} (see also \cite{NOCorr}) and 
carefully analyzing its proof, 
it is possible to
recover the conclusion of
Theorem~\ref{thm:main} as well as 
the formula for $\epsilon_g$ above.
\end{remark}

In the following example, we show that 
Theorem~\ref{thm:main} cannot be generalized to good gradings.

\begin{exmp}
Consider the polynomial ring $\mathbb{C}[x]$ equipped with its natural $\mathbb{Z}$-grading.
Put $\degree(e_{1,1})=\degree(e_{2,2})=0$ and $\degree(e_{1,2})=\degree(e_{2,1})=1.$ Then, using the notation from Example \ref{ex:ABgrading}, the ring $S:=M_2(\mathbb{C}[x])$ is $\mathbb{Z} \times \mathbb{Z}_2$-graded and the grading
is good, but not very good. 
Moreover, this grading is not epsilon-strong, because $S_{(1,0)} S_{(-1,0)} = \{0\}$ while $x e_{1,1} \in S_{(1,0)}$ is nonzero.
\end{exmp}

Now we present a result that allows us to determine when a grading on a matrix ring is very good. 

\begin{prop}\label{simple}
Suppose that $R$ is simple and that 
$S:=M_n(R)$ is $G$-graded. 
Then the $G$-grading on $S$ is very good if 
and only if for every $i \in 
 \overline{n}$ the inclusion 
 $R e_{i,i} \subseteq S_e$ holds.
\end{prop}

\begin{proof}
The ``only if'' statement follows immediately 
from the definition of a very good grading and
the fact that $e_{i,i}$ is a
nonzero idempotent.
Now we show the ``if'' statement.
Take $i,j \in \overline{n}$. 
We claim that $e_{i,j}$ is homogeneous.
If we assume that the claim holds,
then for every $r \in R$ the element
$r e_{i,j}$ is homogeneous since it
belongs to $(R e_{i,i}) e_{i,j}$.
Now we show the claim.
Write $e_{i,j} = \sum_{g \in G} s_g$
for some $s_g \in S_g$ such that
$s_g = 0$ for all but finitely many 
$g \in G$. Since $e_{i,j} = 
e_{i,i} e_{i,j} e_{j,j}$
we may assume that for each $g \in G$
there is $r^{(g)} \in R$ such that
$s_g = r^{(g)} e_{i,j}$. 
From the equality $e_{i,j} = \sum_{g \in G} r^{(g)} e_{i,j}$ it follows that
$1 = \sum_{g \in G} r^{(g)}$ and that
$r^{(g)} \in Z(R)$ for $g \in G$.
Take $h \in G$.
Then $S_h \ni r^{(h)} e_{i,j} = 
\sum_{g \in G} r^{(h)} e_{i,i} 
r^{(g)} e_{i,j}$. Since 
$r^{(h)} e_{i,i} r^{(g)} e_{i,j} 
\subseteq S_g$, for $g \in G$, 
it follows that $r^{(h)} r^{(g)} = 0$,
if $g \neq h$, and 
$(r^{(h)})^2 = r^{(h)}$.
This implies that $R = \oplus_{g \in G}
R r^{(g)}$ as $R$-ideals. Simplicity of
$R$ implies that there is $p \in G$
such that $r^{(g)} = 0$, if $g \neq p$,
and $r^{(p)} = 1$. Thus, $e_{i,j} =
1 e_{i,j} = r^{(p)} e_{i,j} \in S_p$.
\end{proof}

Now we extend \cite[Prop. 2.3]{das99} to matrices with entries in associative unital rings.

\begin{prop}\label{prop:locallystrong}
Suppose that $M_n(R)$ is equipped 
with a very good $G$-grading and write 
$X_i := \{ \degree(e_{i,j}) \}_{j = 1}^n$.
Take $g \in G$.
Then $M_n(R)_g M_n(R)_{g^{-1}} = M_n(R)_e$ if and only if 
$g \in \cap_{i=1}^n X_i$.
\end{prop}

\begin{proof}
Put $S := M_n(R)$. By Theorem~\ref{thm:main}, $S$ is epsilon-strongly $G$-graded. 
It is not difficult to see that 
$S_g S_{g^{-1}} = S_e$ 
if and only 
if $\epsilon_g = 1$. By the proof of 
Theorem~\ref{thm:main} this is
in turn equivalent to $g \in \cap_{i=1}^n X_i$.
\end{proof}

\begin{cor}\label{cor:stronggrading}
Suppose that $M_n(R)$ is equipped with a very good $G$-grading and set
$g_i := \degree(e_{i,1})$ 
for $i \in \overline{n}$. 
Then $M_n(R)$ is strongly $G$-graded if and only if $\overline{g} = G$.
\end{cor}

\begin{proof}
Put $S := M_n(R)$.
By Proposition~\ref{prop:locallystrong}, it follows
that $S$ is strongly $G$-graded if and only if
$G = \cap_{i = 1}^n X_i$. 
This is equivalent to $X_i = G$ 
for every $i \in \overline{n}$.
Since $X_i = \degree(e_{i,1}) X_1$, for 
every $i \in \overline{n}$, this 
is in turn equivalent to 
$(\overline{g})^{-1} = X_1 = G$.
\end{proof}

As before, for every $g \in G$, 
we put $\mathcal{L}_g := \{ i \in \overline{n} \mid
\exists j \in \overline{n}, \ e_{i,j} \in S_g \}$ and
$\mathcal{R}_g := \{ j \in \overline{n} \mid
\exists i \in \overline{n}, \ e_{i,j} \in S_g \}$.
We saw in the proof of Theorem~\ref{thm:main} that $\epsilon_g = \sum_{i \in \mathcal{L}_g} e_{i,i} =
\sum_{j \in \mathcal{R}_{g^{-1}}} e_{j,j}$.

\begin{prop}\label{prop:kerepsilon}
Suppose that $M_n(R)$ is equipped with a very good $G$-grading, and let $g \in G$. With the above notations, the following assertions hold:
\begin{itemize}

\item[(a)] $\ker(\epsilon_g) = V_{G \setminus g \Supp(V)}$

\item[(b)] $\im(\epsilon_g) = 
V_{g \Supp(V)}$

\item[(c)] 
For every
$i \in \overline{n}$ we have
$\epsilon_g(v_i) = v_i$, if $g_i \in g\Supp(V)$, and
$\epsilon_g(v_i) = 0$, otherwise.

\item[(d)] $\epsilon_g|_{V_{g\Supp(V)}} = 
\id_{V_{g\Supp(V)}}$

\end{itemize}
\end{prop}

\begin{proof}
First we prove (a). 
Since $\epsilon_g = \sum_{j \in \mathcal{R}_{g^{-1}}} e_{j,j}$
it follows that 
\begin{align*}
    \ker(\epsilon_g) &= \Span_R \{ v_j \mid 
j \in \overline{n}, \ j \notin \mathcal{R}_{g^{-1}} \} 
= \Span_R \{ v_j \mid j \in \overline{n}, \ 
\forall i \in \overline{n} \ \degree(e_{i,j}) \neq g^{-1} \}\\&
= \Span_R \{ v_j \mid j \in \overline{n}, \ 
\forall i \in \overline{n} \ g_i g_j^{-1} \neq g^{-1} \} 
= \Span_R \{ v_j \mid j \in \overline{n}, \ 
\forall i \in \overline{n} \ g_j \neq g g_i \}.
\end{align*}
Thus, $\ker(\epsilon_g) = V_{ G \setminus g \Supp(V) }$.
Now we prove (b). 
Since $\epsilon_g = \sum_{i \in \mathcal{L}_g} e_{i,i}$
it follows that
\begin{align*}
\im(\epsilon_g)& = \Span_R \{ v_i \mid 
i \in \overline{n}, \ \exists  j \in \overline{n} \
e_{i,j} \in S_g \} 
= \Span_R \{ v_i \mid 
i \in \overline{n}, \ \exists  j \in \overline{n} \
\degree(e_{i,j}) = g \} 
\\&
= \Span_R \{ v_i \mid 
i \in \overline{n}, \ \exists  j \in \overline{n} \
g_i g_j^{-1} = g \} 
= \Span_R \{ v_i \mid 
i \in \overline{n}, \ \exists  j \in \overline{n} \
g_i = g g_j \}.
\end{align*}
Thus, $\im(\epsilon_g) = V_{g \Supp(V)}$.
Statement (c) follows from the definition of $\epsilon_g$ and (a).
Statement (d) follows from (b) and (c).
\end{proof}

\begin{theorem}\label{thm:pre}
Suppose that $M_n(R)$ is equipped with a very good $G$-grading.
For each $g \in G$, put 
$X_g := \{ i \in \overline{n} \mid 
\degree(v_i) = g \}$.
If $|X_g| = |X_e|$, for every $g \in G$
with $X_g \neq \emptyset$, 
then $S$ is an
epsilon-crossed product.
\end{theorem}

\begin{proof}
By Theorem~\ref{thm:main}, $S:=M_n(R)$ is epsilon-strongly $G$-graded.
Suppose that
$|X_g| = |X_e|$, for every $g \in G$
with $X_g \neq \emptyset$.
For all $g,h \in G$ with $X_g$ and $X_h$ nonempty,
fix a bijection $b_{h,g} : X_g \to X_h$
and put $c_{g,h} := b_{h,g}^{-1}$.
Take $g \in G$. We will now define $A_g \in S_g$ and
$B_{g^{-1}} \in S_{g^{-1}}$ in the following way.
Take $i \in \overline{n}$. If 
$g_i \notin g \Supp(V)$, then put 
$B_{g^{-1}}(v_i) := 0$. If $g_i = gh$
for some $h \in \Supp(V)$, then put
$B_{g^{-1}}(v_i) := v_{b_{h,gh}(i)}$.
Take $j \in \overline{n}$. If $g_j \notin g^{-1} \Supp(V)$,
then put $A_g(v_j) := 0$. If $g_j = g^{-1}h$
for some $h \in \Supp(V)$, then put
$A_g(v_j) := v_{ c_{h,g^{-1}h}(j) }$.
It is easy to check that
$A_g B_{g^{-1}} = \epsilon_g$ and 
$B_{g^{-1}} A_g = \epsilon_{g^{-1}}$.
Therefore $S$ is an epsilon-crossed product.
\end{proof}

\begin{lemma}\label{lem:QS}
Suppose that $N \in \N$ and $x_1,\ldots,x_N \in \R$. Then
$
\left( \sum_{i=1}^N x_i \right)^2 \leq N \sum_{i=1}^N x_i^2
$ 
with equality if and only if $x_1 = x_2 = \cdots = x_N$.
\end{lemma}

\begin{proof}
This follows from the Cauchy-Schwarz inequality.
However, in this special case, 
the claim follows immediately from the equality
$
N \sum_{i=1}^N x_i^2 - \left( \sum_{i=1}^N x_i \right)^2 
= \sum_{i < j} (x_i - x_j)^2 
.$
\end{proof}

Notice that the good gradings considered in \cite{das99} are in fact very good, and hence the following result generalizes \cite[Prop.~2.6]{das99} to the context of matrix rings over coefficient rings having invariant basis number (IBN for short). 

\begin{theorem}\label{thm:epsiloncrossed}
Suppose that $R$ has IBN and that $S:=M_n(R)$ is equipped with a very good $G$-grading. 
The following assertions are equivalent:
\begin{itemize}

\item[(i)] $S$ is an epsilon-crossed product;

\item[(ii)] $\Rank(V_g) = \Rank(V_e)$ for every $g\in \Supp(V)$;

\item[(iii)] $n^2 = |\Supp(V)| \cdot \Rank(S_e)$.

\end{itemize} 
\end{theorem}

\begin{proof}
(i)$\Rightarrow$(ii): 
Suppose that $S$ is an epsilon-crossed product.
Take $g \in \Supp(V)$.
Then there are $A_g \in S_g$ and $B_{g^{-1}} \in S_{g^{-1}}$
such that $A_g B_{g^{-1}} = \epsilon_g$ and 
$B_{g^{-1}} A_g = \epsilon_{g^{-1}}$. 
Put $C_g := A_g|_{V_e}$ and $D_{g^{-1}} := B_{g^{-1}}|_{V_g}$.
Then $C_g : V_e \to V_g$ and $D_{g^{-1}} : V_g \to V_e$.
Also, from Proposition~\ref{prop:kerepsilon},
$C_g D_{g^{-1}} = \id_{V_g}$ and 
$D_{g^{-1}} C_g = \id_{V_e}$.
Using that $R$ has IBN,
it follows that $\Rank(V_g) = \Rank(V_e)$.

(ii)$\Rightarrow$(iii):
Suppose that $\Rank(V_g) = \Rank(V_e)$ for every 
$g \in \Supp(V)$. Note that 
\begin{equation}\label{eq:dimVsum}
\Rank(V) = \! \! \! \! \! \sum_{g \in \Supp(V)} 
\! \! \! \! \Rank(V_g)=
\! \! \! \! \! \! \sum_{g \in \Supp(V)} \! \! \! \Rank(V_e) 
= |\Supp(V)| 
\cdot \Rank(V_e) = |\Supp(V)| \cdot \Rank(V_g)
\end{equation}
for every $g \in \Supp(V)$.
 Also note that
\begin{equation}\label{eq:dimSe}
\Rank(S_e) = \Rank( \End_R(V)_e ) =
\sum_{g \in \Supp(V)} \Rank(\End_R(V_g)) =
\sum_{g \in \Supp(V)} (\Rank(V_g))^2.
\end{equation}
Therefore, from 
(\ref{eq:dimVsum}) and (\ref{eq:dimSe}), it follows that
\begin{align*}
n^2 & = (\Rank(V))^2 = 
\sum_{g \in \Supp(V)} \Rank(V) \cdot \Rank(V_g)
\\
& = \sum_{g \in \Supp(V)} |\Supp(V)| \cdot (\Rank(V_g))^2 = 
|\Supp(V)| \cdot \Rank(S_e). 
\end{align*}
(iii)$\Rightarrow$(ii):
Suppose that $n^2 = |\Supp(V)| \cdot \Rank(S_e)$. 
Then, from (\ref{eq:dimSe}), it follows that
\[
 \Biggl(
 \sum_{g \in \Supp(V)} \Rank(V_g) \Biggr)^2 
= n^2 = 
|\Supp(V)| \cdot \sum_{g \in \Supp(V)} (\Rank(V_g))^2.
\]
By Lemma~\ref{lem:QS}, it follows that 
$\Rank(V_g) = \Rank(V_e)$ for all $g \in \Supp(V)$.

(ii)$\Rightarrow$(i):
This follows from Theorem 
\ref{thm:pre}.
\end{proof}

\begin{exmp}
We will now illustrate Theorem
\ref{thm:pre} and Theorem
\ref{thm:epsiloncrossed} with some
concrete examples where 
$S:=M_3(R)$. We will use the 
notation from Theorem~\ref{thm:pre}.

(a) Define a very good 
$\Z_3$-grading on $S$ by
\[S_0:= 
\left[
\begin{array}{ccc}
R & 0 & 0 \\
0 & R & 0 \\
0 & 0 & R
\end{array}
\right], \quad
S_1:=
\left[
\begin{array}{ccc}
0 & R & 0 \\
R & 0 & R \\
0 & R & 0
\end{array}
\right]
\quad \mbox{and} \quad
S_2 := \left[
\begin{array}{ccc}
0 & 0 & R \\
0 & 0 & 0 \\
R & 0 & 0
\end{array}
\right].
\]
Since $|X_0| = |X_1| = |X_2| = 1$,
Theorem~\ref{thm:pre} implies that
$S$ is an epsilon-crossed
product. Indeed, $S$ is even a 
classical $\Z_3$-crossed product.

(b) Define a very good $\Z$-grading on $S$
by putting $S_n:=\{0\}$, 
if $ |n |>2$, and
\begin{displaymath}
	S_0 := \left[\begin{array}{ccc}
	R & 0 & 0 \\
	0 & R & 0 \\
	0 & 0 &  R
\end{array}\right]
\quad
S_1 := \left[\begin{array}{ccc}
	0 & R & 0 \\
	0 & 0 & R \\
	0 & 0 & 0
\end{array}\right]
\quad
S_2 := \left[\begin{array}{ccc}
	0 & 0 & R \\
	0 & 0 & 0 \\
	0 & 0 & 0
\end{array}\right]
\end{displaymath}
and $S_{-n}:=S_n^t,$ 
for $n\in\{1,2\}.$ 
Since $|X_0| = |X_{-1}| = |X_{-2}| = 1$
and $X_n = \emptyset$ for 
$n \in \Z \setminus \{ 0,-1,-2 \}$,
Theorem~\ref{thm:pre} implies that
$S$ is an epsilon-crossed product.
Note that $S$ is not a classical $\Z$-crossed
product since it is not strongly graded
(for instance $S_3 S_{-3} = \{ 0 \}
\neq S_0$).

(c) Define a very good $\Z_2$-grading
on $S$ by putting
\[S_0 := \left[
\begin{array}{ccc}
R & R & 0 \\
R & R & 0 \\
0 & 0 & R
\end{array}
\right]
\quad \mbox{and} \quad
S_1 := \left[
\begin{array}{ccc}
0 & 0 & R \\
0 & 0 & R \\
R & R & 0
\end{array}
\right].
\]
With this grading, $S$ may, or may not
be an epsilon-crossed product,
depending on the choice of $R$.

Case 1: Suppose that $R$ has  IBN.
Since $|X_0|=2 \neq 1 = |X_1|$,
Theorem~\ref{thm:epsiloncrossed}
implies that $S$ is not 
an epsilon-crossed product.

Case 2: Suppose that $K$ is a 
field and that $R$ denotes the 
Leavitt algebra $L_K(1,2)$ over $K$.
Recall that $R$ is the $K$-algebra
generated by four elements 
$x_1,x_2,y_1,y_2$ subject to the 
relations $y_1 x_1 = y_2 x_2 = 1$,
$y_2 x_1 = y_1 x_2 = 0$ and
$x_1 y _1 + x_2 y_2 = 1$.
It is well known that $S$ does 
not have IBN. Furthermore, since
\[ \left[
\begin{array}{ccc}
0 & 0 & y_1 \\
0 & 0 & y_2 \\
x_1 & x_1 & 0
\end{array}
\right]^2
= 
 \left[
\begin{array}{ccc}
y_1 x_1 & y_1 x_2 & 0 \\
y_2 x_1 &  y_2 x_2 & 0\\
0 & 0 & x_1 y _1 + x_2 y_2
\end{array}
\right]
=
 \left[
\begin{array}{ccc}
1 & 0 & 0 \\
0 &  1 & 0\\
0 & 0 & 1
\end{array}
\right]
\]
it follows that $S$ is a $\Z_2$-crossed
product. 
\end{exmp}

\subsection{Good gradings and graded isomorphisms}

Suppose that $S$ and $S'$ are $G$-graded rings.
A ring homomorphism $f : S \to S'$ is 
said to be \emph{graded} if $f(S_g) \subseteq S_g'$ for 
every $g \in G$.
In that case, if $f$ is bijective,
then $f$ is said to a \emph{graded isomorphism}
and $S$ and $S'$ are said to be 
\emph{graded isomorphic}.
The following is clear.

\begin{lemma}\label{esiso} Suppose that $S$ and $S'$ are $G$-graded isomorphic rings. 
If $S$ is epsilon-strongly $G$-graded, then 
$S'$ is epsilon-strongly $G$-graded.
\end{lemma}

 It follows by \cite[Example 1.3]{das99} that good gradings are not preserved by graded isomorphisms. However, it follows by  Theorem~\ref{thm:main} and Lemma~\ref{esiso} that gradings isomorphic to very good gradings are  epsilon-strong. Therefore, by the next result,
 we automatically obtain many new classes of rings endowed with epsilon-strong 
gradings.
 
 \begin{prop}\label{eexm}
 Suppose that $K$ is a field, 
 $p$ is a prime number, and $m,n$ are
 positive integers. 
 Then, under the assumption that we are considering gradings that satisfy $K 1_{M_m(K)} \subseteq M_m(K)_e$, all of the following gradings are epsilon-strong:
\begin{itemize}
    
    \item[(1)] any $C_p$-grading on $M_m(K)$ where $p \nmid m$;
    \item[(2)] any $C_2$-grading on $M_2(K)$ where $\Char(K)\neq 2$ or $K$ is algebraically closed;
    \item[(3)]  any $G$-grading on $M_m(K),$ where $e_{i,j}$ is homogeneous for some $i,j \in \{1,\ldots,n\}$;
    \item[(4)] any $C_n$-grading on $M_m(K)$ 
    if $K$ is algebraically closed and
    $\Char(K) \nmid n$. 
\end{itemize}
 \end{prop}
\begin{proof} (1) Follows by \cite[Prop.~3.2]{BD}, (2) is a consequence of \cite[Cor.~4.2 and Cor.~4.9]{das99}, (3) is obtained from \cite[Cor.~1.6]{das99} and (4) is implied by \cite[Prop.~2.2]{BD}.
\end{proof}

\begin{remark}
    It follows by (3) of Proposition~\ref{eexm}  that  \cite[Example 1.7]{das99} provides an example of an epsilon-strong $C_2$-grading on $M_2(K)$ in which no $e_{i,j}$ is homogeneous.
\end{remark}

\end{document}